\documentclass{article}
\usepackage{amsmath,amssymb, wasysym,  amsmath,amsfonts, amscd}
\newcommand{\calf}{{\cal F}}
\newcommand{\cals}{{\cal S}}
\newcommand{\calt}{{\cal T}}
\newcommand{\calu}{{\cal U}}
\newtheorem{lem}{Lemma}

\newenvironment{proof}{\noindent{\bf Proof.\,}}{\hfill$\Box$~\\}
\newenvironment{proofcite}[1]{\noindent{\bf Proof of #1.\,}}{\hfill$\Box$~\\}

\input epsf

\newcommand{\cala}{{\mathcal{A}}}

\newcommand{\calx}{{\mathcal{X}}}
\newcommand{\calg}{{\mathcal{G}}}

\newcommand{\calc}{{\cal C}}
\newcommand{\qed}{{\hfill$\Box$}}

\newcommand{\betai}{\beta_i}

\newcommand{\betazero}{\beta_0}
\newcommand{\Obetazero}{\overline{\beta}_0}

\newcommand{\betaone}{\beta_1}
\newcommand{\Obetaone}{\overline{\beta}_1}

\newcommand{\Obeta}{\overline{\beta}}

\newcommand{\betatwo}{\beta_2}

\newcommand{\betapar}{\beta_{\|}}
\newcommand{\Obetapar}{\overline{\beta}_{\|}}

\newcommand{\betawedge}{\beta_{\wedge}}
\newcommand{\Obetawedge}{\overline{\beta}_{\wedge}}

\newcommand{\betathree}{\beta_3}

\newcommand{\betatri}{\beta_{\triangle}}
\newcommand{\Obetatri}{\overline{\beta}_{\triangle}}

\newcommand{\betaw}{\beta_{\dashv}}
\newcommand{\Obetaw}{\overline{\beta}_{\dashv}}

\newcommand{\betapath}{\beta_{\sqcup}}
\newcommand{\Obetapath}{\overline{\beta}_{\sqcup}}

\newcommand{\alphaone}{\alpha_1}

\newcommand{\ove}{\overline{e}}

\newcommand{\ds}{\displaystyle}

\newtheorem{thm}{Theorem}

\title{$Q_{2}$-free families in the Boolean lattice}
\author{Maria Axenovich\thanks{Supported in part by NSA grant H98230-09-1-0063 and NSF grant DMS-0901008.}\\ Iowa State University\\ Ames, IA 50010\\ \texttt{axenovic@iastate.edu}\and
Jacob Manske\\ Iowa State University\\ Ames, IA 50010\\ \texttt{jmanske@iastate.edu}\and
Ryan Martin\thanks{Supported in part by NSA grant H98230-08-1-0015 and NSF grant DMS-0901008.}\\
Iowa State University\\ Ames, IA 50010\\ \texttt{rymartin@iastate.edu}}

\begin{document}
\maketitle

\begin{abstract}
For a family $\calf$  of subsets of $[n]=\{1, 2, \ldots, n\}$ ordered by inclusion,  and a partially ordered set $P$, we say that $\calf$ is $P$-free if it does not
contain a subposet isomorphic to $P$. Let $ex(n, P)$ be the largest size of a $P$-free family of subsets of $[n]$.
Let $Q_2$ be the poset with distinct elements $a, b, c, d$,  $a<b,c<d$; i.e.,  the $2$-dimensional Boolean lattice.
We  show that $2N -o(N) \leq ex(n, Q_2)\leq   2.283261N +o(N), $  where $N = \binom{n}{\lfloor n/2 \rfloor}$.
We also prove that the largest $Q_{2}$-free family of subsets of $[n]$ having  at  most three different   sizes  has at most $2.20711N$ members.
\end{abstract}

\section{Introduction}
\noindent
Let $Q_n$ be the $n$-dimensional Boolean lattice corresponding to subsets of an $n$-element set ordered by inclusion.
A poset $P = (X, \leq)$ is a subposet of $Q = (Y,\leq')$  if there is an injective map $f: X \rightarrow Y$ such that for $x_{1},x_{2} \in X$, $x_{1} \leq x_{2}$ implies $f(x_{1})\leq' f(x_{2})$.
For a poset $P$, we say that a set of elements $\calf \subseteq 2^{[n]}$ is $P$-free if   $(\calf, \subseteq)$  does not contain $P$ as a subposet.
Let $ex(n, P)$ be the size of the largest $P$-free family of subsets of $[n]$.
We say that  the set of all $i$-element subsets of $[n]$, $\binom{[n]}{i}$,  is the $i$th layer of $Q_n$.  Finally, let $N(n) = N = \binom{n}{\lfloor n/2 \rfloor }$;  i.e.,  $N$ is  the size of a middle layer of a Boolean lattice.

\noindent
The classical theorem of Sperner \cite{sperner28} says that $ex(n, Q_1) = N$.  Most  asymptotic bounds for $ex(n,P)$ are expressed in terms of $N$.
Many largest $P$-free families are simply unions of largest layers in $Q_n$.
For example,  Erd\H{o}s generalized Sperner's result in \cite{erdos45-sperner}, showing that the size of the largest subposet of  $Q_{n}$ which does not contain
a chain with $k$ elements, $C_k$,  is equal to the number of elements in the $k-1$ largest layers of $Q_n$;  i.e., for a fixed $k$, $ex(n, C_k) = (k-1)N+o(N)$.
De Bonis, Katona and Swanepoel showed  in \cite{debkatswan05} that $ex(n,\Bowtie) = 2N+o(N)$, where $\Bowtie$ is a subposet of $Q_n$ consisting of distinct sets $a, b, c, d$ such that
$a ,b \subset c, d$.
 De Bonis and Katona, as well as Thanh showed  in \cite{debkat2007}, \cite{Thanh98} that
$ex(n, V_{r+1}) = N +  o(N)$, where $V_{r+1}$ is a subposet of $Q_n$ with distinct elements $f, g_i$, $i=1,\ldots, r$,
 $f \subset g_{i}$ for $i = 1, \ldots, r$.  More generally, for a poset $K_{s,t}$,  with distinct elements $f_1, \ldots, f_s \subset g_1, \ldots, g_t$,  and
 a poset $P_k(s)$, with distinct elements $f_1 \subset \cdots \subset f_k \subset g_1, g_2, \ldots, g_s$,
  Katona and Tarjan  \cite{katonatarjan83}   and later De Bonis and Katona \cite{debkat2007} proved  that $ex(n, K_{s, t}) = 2N+o(N)$ and $ex(n, P_k(s)) = kN+o(N)$, respectively.
Griggs and Katona proved  in \cite{griggskatona08} that
$ex(n, \textsf{N}) = N + o(N)$, for a poset $\textsf{N}$ with distinct elements $a,b,c,d$, such that  $a\subset c,d$,  and $b\subset c$.
Griggs and Lu \cite{griggslu09} proved that $ex(n,P_k(s,t)) = (k-1)N+o(N)$, where $P_k(s, t)$ is a poset with distinct elements
$f_1, f_2 \ldots, f_s \subset g_2\subset g_3 \subset  \cdots \subset  g_{k-1}\subset   h_1, \ldots, h_t$, $k\geq 3$.
 They also showed that $ex(n, O_{4k}) =N+o(N)$,   $ex(n, O_{4k-2})\leq (1+ \sqrt{2}/2) N+o(N)$, where $O_i$ is a poset of height two
 which is a cycle of length $i$ as an undirected graph.  More generally, they proved that if $G=(V,E)$ is a graph and $P$ is a poset with elements $V\cup E$,
 with $v<e$ if $v\in V$, $e\in E$ and $v$ incident to $e$, then
 $ex(n,P) \leq (1- \sqrt {1 - 1/(\chi(G)-1)})N+o(N)$.
 Bukh \cite{bukh08} proved that $ex(n,T)= kN+o(N)$, where $T$ is a poset whose diagram is a tree and $k$ is an integer which is one less than the height of $T$.
 As a general reference in poset theory, see \cite{trotterbook}.

 \noindent
The smallest  poset, $P$, for which $ex(n,P)$ is not known to be an integer multiple of $N$,  is $P=Q_2$.
 This manuscript is devoted to this little poset for which we still do not know whether $ex(n, Q_2)= kN+o(N)$ for an integer $k$.
 We show that $ 2N-o(N)\leq ex(n, Q_2)   \leq   2.283261N+o(N)$. We believe that $ex(n, Q_2) = 2N+o(N)$.
 Next, are our main results.

 \begin{thm}\label{mainthm} If $\calf \subset Q_{n}$ is $Q_{2}$-free, then
$2N- o(N) \leq |\calf| \leq 2.283261N+ o(N)$.
\end{thm}

\begin{thm}\label{consecthm}
Let $\calf \subset Q_{n}$ be a  $Q_{2}$-free family, $\calf= \cals\cup \calt \cup \calu$,   where $\cals$ is a subset of minimal elements of $\calf$,
$\calu$ is a subset of maximal elements of $\calf$ and $\calt = \calf\setminus (\cals\cup \calu) $ such that for any $T\in \calt$, $S\in \cals$, $U\in \calu$,  $|T| =k$, $|U|>k$, $|S|<k$.
Then $|\calf| \leq  N(3+\sqrt{2})/2 + o(N) \leq  2.20711N+o(N)$.
In particular, if $\calf$ is a $Q_2$-free  subset of three layers of $Q_n$,  then  $|\calf| \leq 2.20711N+o(N)$.
\end{thm}

\noindent
We  prove the main theorems  in Sections \ref{PROOF1} and \ref{PROOF2},
prove  supporting lemmas  in Section \ref{LEMMAS}.


\section{Proof of  Theorem \ref{mainthm} }\label{PROOF1}

{\bf Sketch of the proof.}
We consider a $Q_2$-free family, $\calf$,  of subsets of $[n]$. Using a standard argument, we assume that all members of $\calf$ have  size between $n/2-n^{2/3}$ and $n/2+n^{2/3}$.
We bound $\calf$ in terms of the number of  full chains containing exactly $3$ sets  or exactly  $1$ set of $\calf$.
In doing this, we introduce an auxiliary graph corresponding to $2$-element subsets in local  sub-lattices, express the number of chains in terms of the size of that graph, and
optimize the resulting expression.  This produces the upper bound in the statement of the theorem.   The lower bound is achieved by $\calf = \binom{[n]}{\lfloor n/2\rfloor} \cup  \binom{[n]}{\lfloor n/2 \rfloor +1}$.

\begin{figure}\label{f1}
\begin{center}
\epsffile{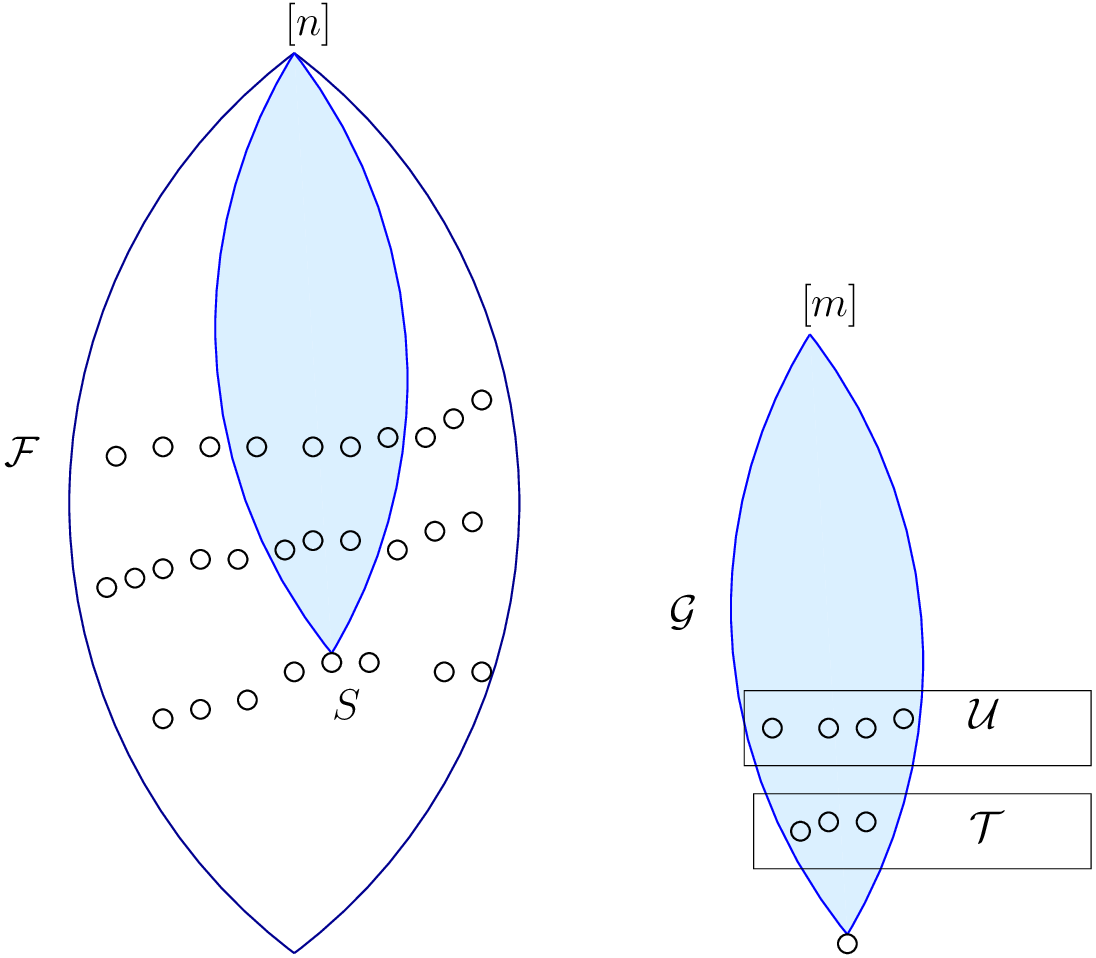}
\end{center}
\caption{Local argument}
\end{figure}

\noindent
Let us now begin the proof in full. Let $\calf $ be a $Q_2$-free family of subsets of $[n]$,  let $\cals$ be the set of  minimal elements of $\calf$.
\begin{lem}\label{lem:mid} $ \sum_{|k-n/2|\geq n^{2/3}}\binom{n}{k}\leq 2^{n-\Omega(n^{1/3})}=2^{-\Omega(n^{1/3})}N . $
\end{lem}

\begin{proofcite}{Lemma~\ref{lem:mid}}
We note that the expression $2^{-n}\sum_{|k-n/2|\geq n^{2/3}}\binom{n}{k}$ computes the probability that a $B(n,1/2)$ binomial random variable, $X$,  takes on values outside of the interval $\left(n/2-n^{2/3},n/2+n^{2/3}\right)$.  Using a standard Chernoff bound,
$\Pr\left(|X-n/2|\geq \delta(n/2)\right)\leq 2\exp\left\{-(n/2)\delta^2/2\right\}$.  Observing that the left-hand side sums $\binom{n}{k}$ over all $k$ for which $|k-n/2|\geq \delta(n/2)$ and setting $\delta=2n^{-1/3}$, we can conclude: \\
$$ \sum_{|k-n/2|\geq n^{2/3}}\binom{n}{k}\leq 2^{n+1}e^{-n^{1/3}} . $$

\noindent
Since $\binom{n}{n/2}=\Omega(n^{-1/2})2^n$, we may conclude that
$ \sum_{|k-n/2|\geq n^{2/3}}\binom{n}{k}\leq 2^{-\Omega(n^{1/3})}\binom{n}{n/2} . $
Note that, for every $C$ there exists a $c$ such that $\sum_{|k-n/2|\geq cn^{1/2}\ln n}\binom{n}{k}\leq n^{-C}\binom{n}{n/2}$.  So, we could in fact  have chosen a more precise
$\Omega(n^{1/2}\ln n)$ as our error term, rather than the more convenient $n^{2/3}$.
\end{proofcite}

\noindent
As a result of Lemma~\ref{lem:mid}, we can assume  that all elements in $\calf$ are close to the middle layer. A full chain in $Q_n$ is a chain containing $n+1$ sets. For $i=1,2,3$ and a set $F\in\calf$, let $\Upsilon^i_n(F, \calf)$ denote the set of full chains in $Q_n$ that contain $F$ and exactly $i-1$ other members of $\calf$. Let $ \Upsilon^i_n( \calf)$ be the set of full chains in $Q_n$ that contain exactly $i$  members of $\calf$.

\begin{lem} \label{lem:chains}    $|\calf| \leq
   \left( 2+\frac{1}{n!} \left(|\Upsilon_n^3(\calf)|-|\Upsilon_n^1(\calf) |\right) \right)N
   \leq    \left( 2+\frac{1}{n!} \sum _{S\in \cals } \left(|\Upsilon_n^3(S, \calf)|-|\Upsilon_n^1(S, \calf) |\right) \right)N. $
\end{lem}

\begin{proofcite}{Lemma \ref{lem:chains}}
\noindent
Let $\Upsilon$ be the set of all full chains in $Q_n$. 
Let $\calx= \{ (F, \sigma):  F\in \calf, \sigma\in \Upsilon, F\in \sigma \}.$
Since each $\sigma \in \Upsilon$ contains at most $3$ sets from $\calf$, we have that
$$|\calx| = 3 |\Upsilon_n^3(\calf)| + 2 |\Upsilon_n^2(\calf)| + |\Upsilon_n^1(\calf)|.$$
On the other hand, any $F\in \calf$ is contained in $|F|!(n-|F|)!\geq \lfloor n/2 \rfloor !   \lceil n/2 \rceil !$ full chains from $\Upsilon$.
Thus   $ |\calf| \lfloor n/2 \rfloor !   \lceil n/2 \rceil !  \leq |\calx| = 3 |\Upsilon_n^3(\calf)| + 2| \Upsilon_n^2(\calf)| + |\Upsilon_n^1(\calf)|.$

\noindent
Since the terms $|\Upsilon_n^i(\calf)|$ sum to $n!$,
$|\calx|  = 2n! + |\Upsilon_n^3(\calf)|-|\Upsilon_n^1(\calf)|$.
Thus,
\begin{eqnarray*}
\lfloor n/2 \rfloor !   \lceil n/2 \rceil! |\calf | & \leq  &   2n! + |\Upsilon_n^3(\calf)| - |\Upsilon_n^1(\calf)|,\\
|\calf| &\leq &   2N+ \frac{1}{\lfloor n/2 \rfloor !   \lceil n/2 \rceil! } ( |\Upsilon_n^3(\calf)| - |\Upsilon_n^1(\calf)| )   =  \left(2 + \frac{1}{n!} \left( |\Upsilon_n^3(\calf)| - |\Upsilon_n^1(\calf)| \right)\right) N.
\end{eqnarray*}
The second inequality in the lemma follows from the fact that every member of $\Upsilon_n^3(\calf)$ contains a member of $\cals$.
\end{proofcite}

\noindent
Fix  $S\in \cals$. We shall bound  $\left(|\Upsilon_n^3(S, \calf)|-|\Upsilon_n^1(S, \calf) |\right).$ Let  $\calg= \calg(S) = \{ F\setminus S:  F\in \calf, S\subseteq F\}.$
We see that $\calg$ is a system of subsets of an $m$-element set,  where $m= n-|S|$,  see Figure 1.
Moreover, $\emptyset \in \calg$, and since $\calf$ is $Q_2$-free,  for any $X\in \calg$, there is at most one set $Y\in \calg \setminus \emptyset$, such that $Y \subseteq X$.
We see also that  $|\Upsilon_n^i(S, \calf)| = |S|! |\Upsilon_{m}^i(\emptyset, \calg)|.$

\noindent
Let $\calt$ be the set of minimal elements in  $\calg-\{\emptyset\}$ and $\calu = \calg -(\calt \cup \{\emptyset\})$.
Let $\calt_i = \{ T \in \calt:   |T|=i\}$, $i=1, 2, 3, \ldots$.
Without loss of generality let $\calt_1 = \{\{\eta+1\}, \{\eta+2\}, \ldots, \{m\}\}$,   as a result  $\calt_2 $ is a set of some two-element  subsets of $[\eta]$.
We create an auxiliary graph $G$ corresponding to $\calg$ by letting the vertex set be $[\eta]$ and the edge set be $\calt_2$.
See Figure \ref{f2} for illustration.
Let $e, \ove$ be the number of edges and nonedges in $G$, respectively.
Let   $\Upsilon_i =  \Upsilon_m^i (\emptyset, \calg)$, $i=1,3$.

\begin{figure}\label{f2}
\begin{center}
\epsffile{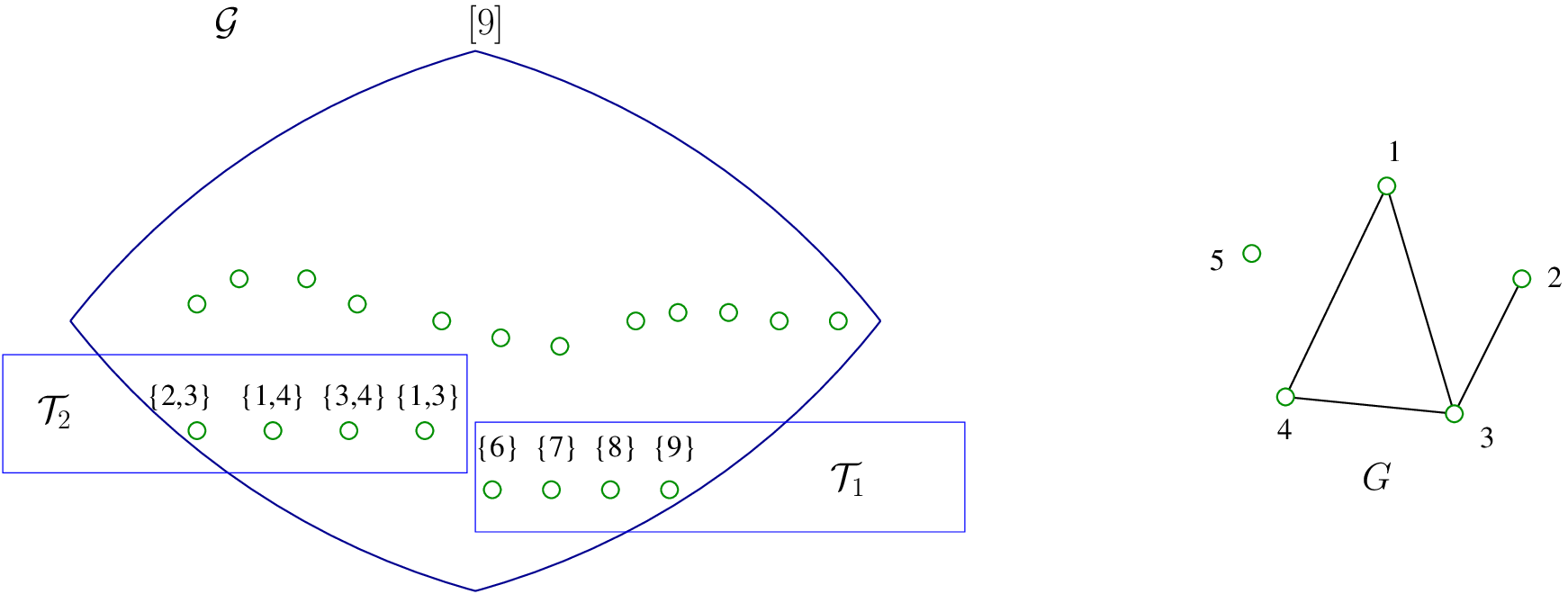}
\end{center}
\caption{Family $\calg$ and graph $G$, $m=9$, $\eta=5$.}
\end{figure}

\noindent
We shall then express the bounds on $|\Upsilon_1|$ and $|\Upsilon_3|$ in terms of proportions $a= \eta/m $ and
$b = \frac{e}{\binom{\eta}{2}}$.
Note that $ 0\leq a,b\leq 1$.
Finally, set \begin{equation}
 \mu =
 \begin{cases}
 1,& a< 1/2\\
 \frac{1-a}{a}, &   a\geq 1/2.
 \end{cases}\nonumber
 \end{equation}

\noindent
Next, we state the technical lemmas which are proved in Section \ref{LEMMAS}.

\begin{lem}\label{lem:sig1}  $  |\Upsilon_1|\geq      m!\left[   b(a^3-a^2)\mu  + (a^2-a^3)\mu  + O(m^{-1}) \right] . $
\end{lem}

\begin{lem} \label{sigma3} $  |\Upsilon_3| \leq m! \left [ b^2 \left(      a^4/2- a^3  \right)   + b\left(    a^3 - 3a^4/4 \right) + \left(  a^4/4 -a^2+a    \right)      + O(m^{-1})   \right].$
  \end{lem}

\noindent
With Lemmas~\ref{lem:sig1} and~\ref{sigma3},
$$ |\Upsilon_3|-|\Upsilon_1|\leq m!\left[b^2\left(a^4/2-a^3\right)+b\left(a^3-3a^4/4-a^3\mu+a^2\mu\right) +\left(a^4/4-a^2+a-a^2\mu+a^3\mu\right) +O(m^{-1})\right] . $$

\begin{lem} \label{lem:calc}
With $0\leq a,b\leq 1$ and $\mu=\mu(a)$ as defined above,
$$ b^2\left(a^4/2-a^3\right)+b\left(a^3-3a^4/4-a^3\mu+a^2\mu\right)+\left(a^4/4-a^2+a-a^2\mu+a^3\mu\right)\leq 0.283261 . $$
\end{lem}

\noindent
Using Lemma~\ref{lem:calc},
\begin{equation}\label{eq:upsbound}
   |\Upsilon_m^3(\emptyset, \calg)| - |\Upsilon_m^1(\emptyset, \calg)| \leq \left[0.283261 + O(m^{-1})\right]m!.
\end{equation}
For a final calculation, we need the so-called LYM inequality, proven by   Yamamoto   \cite{yamamotolym},   Bollob\'as  \cite{bollobaslym},  Lubell  \cite{lubelllym}, and  Meshalkin \cite{meshalkinlym}.
\begin{lem} [LYM inequality]
If $\cala$ is an antichain in $Q_n$, then $\ds\sum_{A \in \cala}\ds\binom{n}{|A|}^{-1} \leq 1.$
\end{lem}

\noindent
Returning to Lemma \ref{lem:chains}, we have
$$ |\calf|\leq N\left( 2+\frac{1}{n!} \sum _{S\in \cals } \left(|\Upsilon_n^3(S, \calf)|-|\Upsilon_n^1(S, \calf) |\right) \right) . $$
Using inequality (\ref{eq:upsbound}), we have
$$ |\calf|\leq N\left( 2+ \sum _{S\in \cals } \frac{1}{n!}   |S|!\left[  0.283261  + O((n-|S|)^{-1})) \right] (n-|S|)!  \right) . $$
LYM and the fact that that $(n-|S|)^{-1}\leq (n/2 - n^{2/3} )^{-1}$ give
$$ |\calf|\leq N\left( 2.283261+ O((n/2 - n^{2/3} )^{-1})\right)= 2.283261N+o(N) . $$

\noindent
This concludes the proof of the main theorem.


\section{Proof of Theorem \ref{consecthm}}\label{PROOF2}
\newcommand{\Nk}{N'}
\noindent
For ease of notation, in this proof let $\Nk=\binom{n}{k}$.  Suppose $\calf$ is a $Q_{2}$-free family from 3 layers, $L_1, L_2, L_3$,  of the Boolean lattice $Q_n$, where $L_1=\binom{[n]}{k-1}$,  $L_2=\binom{[n]}{k}$, and $L_3=\binom{[n]}{k+1}$. Let $\cals = \calf \cap L_{1}$, $\calt = \calf \cap L_{2}$, and $\calu = \calf \cap L_{3}$.  We may assume that $|k-n/2|<n^{2/3}$ as a result of Lemma~\ref{lem:mid}.  Furthermore, it will be useful to assume that $|\cals|,|\calu|\leq\Nk$; otherwise we could delete at most $\binom{n}{k-1}-\Nk=O(n^{-1/3})\Nk=o(\Nk)$ members of $\cals$ and $O(n^{-1/3})\Nk=o(\Nk)$ members of $\calu$ to ensure that the resulting sets are at most $\Nk$.
Let $\Upsilon$ be the set of $3$-element chains contained in $L_1\cup L_2\cup L_3$,   and
$\Upsilon_i = \{ \sigma \in \Upsilon:  |\sigma\cap \calf| = i\}$, $i= 0, 1, 2, 3$.

\noindent
We count ordered pairs, one element is a member of $\calf$ and the other is a chain from $\Upsilon$.  That is, $\calx:=\{(F,\sigma) : F\in\calf, \sigma\in\Upsilon, F\in\sigma\}$.
Then
$$|\calx| = 3 |\Upsilon_3| + 2|\Upsilon_2| + |\Upsilon_1| = 2|\Upsilon| + |\Upsilon_3|-|\Upsilon_1| - 2|\Upsilon_0|  \leq  2|\Upsilon| + |\Upsilon_3|-|\Upsilon_1|.$$
On the other hand,
$$|\calx| = (k+1)k |\calu| + k(n-k)|\calt| + (n-k+1)(n-k) |\cals|.$$
 Putting together these  expressions for  $|\calx|$ and using the fact that
$|\Upsilon|  = \Nk k(n-k) $, we have
\begin{equation}\label{main-ineq}
(k+1)k |\calu| + k(n-k)|\calt| + (n-k+1)(n-k) |\cals|\leq   2\Nk k(n-k) + |\Upsilon_3|-|\Upsilon_1|.
\end{equation}

\noindent
For $X \in L_{1}$, $Y \in L_{2}$, and $Z \in L_{3}$, define
\begin{eqnarray}
f(X)=  |\{ T\in \calt:   X\subset T\}|;\  \quad  g(Z) = |\{ T\in \calt:   Z\supset T\}|; \nonumber \\
 \breve{f}(Y)=|\{ S\in \cals:   S\subset Y\}|;\  \quad
\breve{g}(Y)=|\{ U\in \calu:   U\supset Y\}|. \nonumber
\end{eqnarray}
Note that $\ds\sum_{X\in\cals}f(X)=\ds\sum_{Y\in\calt}\breve{f}(Y)  \quad \mbox{  and} \quad  \ds\sum_{Z\in\calu}g(Z)=\ds\sum_{Y\in\calt}\breve{g}(Y).$

\noindent
We shall bound $|\Upsilon_3|-|\Upsilon_1|$ by counting the chains that contain an element of $\calt$, $\cals$ and $\calu$, then chains containing an element of $\calt$, $L_1\setminus \cals$, $L_3\setminus \calu$.

\begin{eqnarray}
   |\Upsilon_3|-|\Upsilon_{1}|  & \leq & \sum_{Y\in\calt}\left[\breve{f}(Y)\breve{g}(Y)
   - \left(k-\breve{f}(Y)\right)\left(n-k-\breve{g}(Y)\right)\right]  \nonumber\\
   & = & (n-k)\sum_{X\in\cals}f(X)+k\sum_{Z\in\calu}g(Z)-|\calt| k(n-k). \label{upsilons}
\end{eqnarray}

\noindent
Now, we shall find a bound on  $\sum f$ and $\sum g$ in terms of $|\cals|$ and $|\calu|$.  Recall that we were able to assume that $|\cals|,|\calu|\leq\Nk$.

\begin{lem}\label{sumf}
If $\Nk=\binom{n}{k}$ and $|k-n/2|=O(n^{2/3})$, then
\begin{eqnarray*}
\sum_{X\in\cals}f(X)\leq (k+1)\sqrt{|\cals|\left(\Nk-|\calu|\right)}+O(n^{5/6}) \Nk, \\
\sum_{Z\in\calu}g(Z)\leq (n-k+1)\sqrt{|\calu|\left(\Nk-|\cals|\right)}+O(n^{5/6}) \Nk.
\end{eqnarray*}
\end{lem}

\begin{proofcite}{Lemma~\ref{sumf}}
Consider any $X\in L_1$.  One can associate members of $L_2$ lying above $X$ with the elements of $[n]-X$ that they contain.  Furthermore, one can associate members of $L_3$ lying above $X$ with the pairs of elements of $[n]-X$ that they contain.  So, for any $X\in L_1$, then $|\{U\in\calu : U\supset X\}|\leq\binom{n-k+1}{2}$.  But if $X\in\cals$, there are $\binom{f(X)}{2}$ members of $L_3$ that cannot be above $X$.  Hence, for each $X\in\cals$,
$$ |\{U\in\calu: U\supset X\}| \leq \binom{n-k+1}{2}-\binom{f(X)}{2} . $$
Symmetrically, for any $Z\in L_3$, $|\{S\in\cals : S\subset Z\}|\leq\binom{k+1}{2}$ but if $Z\in\calu$, then
$$ |\{S\in\cals: S\subset Z\}|\leq\binom{k+1}{2}-\binom{g(Z)}{2} . $$

\noindent
Now we double-count the pairs $(X,U)$ such that $X\in L_1$, $U\in\calu$ and $X\subset U$:
$$ \binom{k+1}{2}|\calu|=\sum_{X\in L_1}|\{U\in\calu : U\supset X\}| . $$
We can partition the members of $X\in L_1$ according to whether or not $X\in\cals$ and use the estimates above.  To wit,
\begin{eqnarray*}
   |\calu| & \leq & \binom{k+1}{2}^{-1}  \left(\sum_{X\in \cals} \left(\binom{n-k+1}{2} - \binom{f(X)}{2}\right) + \sum_{X\in L_1-\cals}   \binom{n-k+1}{2}\right) \\
   & = & \binom{k+1}{2}^{-1}\left(|L_1|\binom{n-k+1}{2}-\sum_{X\in\cals}\binom{f(X)}{2}\right) .
\end{eqnarray*}
Since $|L_1|=\binom{n}{k-1}$, then the first term simplifies to $\binom{n}{k+1}$.  Hence,
$$ |\calu|\leq\binom{n}{k+1}-\frac{1}{(k+1)_2}\sum_{X\in\cals}(f(X))_2 . $$

\noindent
Jensen's inequality allows us to bound $\sum_{X\in\cals}f^2(X)\geq\frac{1}{|\cals|}\left(\sum_{X\in\cals}f(X)\right)^2$.  Furthermore, since $f(X)\leq n-k+1$ and $|S|\leq\binom{n}{k-1}$,  $\sum_{X\in\cals}f(X)\leq\binom{n}{k-1}(n-k+1)$.
\begin{eqnarray*}
   |\calu| & \leq & \binom{n}{k+1}-\frac{1}{(k+1)_2|\cals|}\left(\sum_{X\in\cals}f(X)\right)^2+\binom{n}{k-1}\frac{n-k+1}{(k+1)_2} \\
   & = & \Nk\frac{n-k+1}{k+1}-\frac{1}{(k+1)_2|\cals|}\left(\sum_{X\in\cals}f(X)\right)^2 .
\end{eqnarray*}

\noindent
Rearranging the terms gives
$$ \left(\sum_{X\in\cals}f(X)\right)^2\leq (k+1)_2|\cals|\left(\Nk \frac{n-k+1}{k+1}-|\calu|\right) . $$
Now, we solve for the summation and make some easy estimates:
\begin{eqnarray*}
   \sum_{X\in\cals}f(X) & \leq & (k+1)\sqrt{|\cals|\left(\Nk-|\calu|\right)+|\cals|\frac{n-2k}{k+1} \Nk} \\
   & \leq & (k+1)\sqrt{|\cals|\left(\Nk-|\calu|\right)}+(k+1)\sqrt{|\cals|\frac{|n-2k|}{k+1} \Nk} \\
   & \leq & (k+1)\sqrt{|\cals|\left(\Nk-|\calu|\right)}+O(n^{5/6}) \Nk
\end{eqnarray*}

\noindent
Symmetrically, $\sum_{Z\in\calu}g(Z)\leq (n-k+1)\sqrt{|\calu|\left(\Nk-|\cals|\right)}+O(n^{5/6}) \Nk$, and this concludes the proof of Lemma~\ref{sumf}.
\end{proofcite}

\noindent
Returning to (\ref{main-ineq}) and using (\ref{upsilons}) we have:
\begin{eqnarray}
   \lefteqn{\hspace{-1in}(k+1)k |\calu| + k(n-k)|\calt| + (n-k+1)(n-k) |\cals|} \nonumber \\
   & \leq & 2 \binom{n}{k}  k(n-k) +  |\Upsilon_3| - |\Upsilon_1| \nonumber \\
   & \leq & 2\Nk k(n-k)+ (n-k)\sum_{X\in\cals}f(X)+k\sum_{Z\in\calu}g(Z)-|\calt| k(n-k) . \label{eq:kn_k}
\end{eqnarray}
As $|k-n/2|=O(n^{2/3})$ , we can utilize the estimates in Lemma \ref{sumf} to bound $\sum_{X\in\cals}f(X)$ and $\sum_{Z\in\calu}g(Z)$ and divide (\ref{eq:kn_k}) by $k(n-k)$ to get
$$ \frac{k+1}{n-k}|\calu|+|\calt|+\frac{n-k+1}{k}|\cals|
   \leq 2\Nk +\frac{k+1}{k}\sqrt{|\cals|\left(\Nk-|\calu|\right)} +\frac{n-k+1}{n-k}\sqrt{|\calu|\left(\Nk-|\cals|\right)}+O(n^{-1/6}) \Nk-|\calt| . $$

\noindent
The goal is to get $2|\calf|$ on the left-hand side of the inequality. What this enables us to do is to eliminate $|\calt|$ from the right-hand side.  We may disregard all small-order terms because they are of magnitude at most $O(n^{-1/6}) \Nk$:
\begin{eqnarray*}
   2|\calu|+2|\calt|+2|\cals| & \leq & 2\Nk +|\calu| +|\cals| +\sqrt{|\cals|\left(\Nk-|\calu|\right)} +\sqrt{|\calu|\left(\Nk-|\cals|\right)}+O(n^{-1/6}) \Nk \\
   & \leq & \frac{3+\sqrt{2}}{2}\Nk+O(n^{-1/6}) \Nk .
\end{eqnarray*}

\noindent
Here the last inequality is  obtained by maximizing function $f(u,s)=2+\sqrt{s(1-u)}+\sqrt{u(1-s)}+u+s$, $0\leq u,s\leq 1$. The maximum occurs when $s=u=(2+\sqrt{2})/4$.

\noindent
Therefore,
\begin{equation}\label{eq:nkbound}
   |\calf|\leq\frac{3+\sqrt{2}}{2}\Nk+o(\Nk)\leq 2.20711\Nk+o(\Nk) .
\end{equation}

\noindent
Consider now a more general setting. Recall that $N=\binom{n}{n/2}$.  Let $\calf$ be a $Q_2$-free family of sets in $Q_n$.
Let $\calf = \cals \cup \calt \cup \calu$, where $\cals\subset\binom{[n]}{k_\cals}$,  $\calt\subset\binom{[n]}{k}$ and $\calu\subset\binom{[n]}{k_\calu}$, where $k_\cals<k<k_\calu$.  We may assume that $n/2-n^{2/3}<k_\cals<k<k_\calu<n/2+n^{2/3}$.  Otherwise, by Lemma~\ref{lem:mid}, at least one of $\cals$, $\calt$ or $\calu$ has size $o(N)$ and so $|\calf|\leq (2+o(1))N$.

\noindent
Consider a Symmetric Chain Decomposition of $Q_n$ (see Greene and Kleitman~\cite{GK} for the existence of such a decomposition) and, in particular, the $\Nk$ disjoint chains that contain elements of $\binom{[n]}{k}$, call them $P_1,\ldots,P_{\Nk}$.  We can create a new family $\calf'=\calu'\cup\calt\cup\cals'$ such that we shift $\cals$ and $\calu$ to the layers directly below and above $\calt$, respectively, along each chain  $P_i$. Formally, let
\begin{eqnarray*}\
   \cals' & = & \left\{P_i\cap\binom{[n]}{k-1} : \mbox{ there is }S\in\cals\cap P_i, i=1,\ldots,q\right\} , \\
   \calu' & = & \left\{P_i\cap\binom{[n]}{k+1} : \mbox{ there is }U\in\calu\cap P_i, i=1,\ldots,q\right\} .
\end{eqnarray*}

\noindent
Note that $\calf'$ is $Q_2$-free and consists of three consecutive layers.  Thus, the inequality (\ref{eq:nkbound}) gives that $|\calf'|\leq \frac{3+\sqrt{2}}{2}\Nk+o(\Nk)$.

\noindent
There might be unshifted elements, but not too many.  In fact, both $|\cals|-|\cals'|$ and $|\calu|-|\calu'|$ are at most $N-\Nk$.  So,
\begin{eqnarray*}
   |\calf| & = & |\calf'|+\left(|\cals|-|\cals'|\right)+\left(|\calu|-|\calu'|\right) \\
   & \leq & \frac{3+\sqrt{2}}{2}\Nk+o(\Nk)+2(N-\Nk) \\
   & = & \left(\frac{\sqrt{2}-1}{2}\right)\Nk+2N+o(N) \\
   & \leq & \left(\frac{\sqrt{2}-1}{2}\right)N+2N+o(N) \\
   & \leq & \frac{3+\sqrt{2}}{2} N+o(N)\approx 2.20711 N+o(N) .
\end{eqnarray*}\qed


\section{Proofs of Lemmas}\label{LEMMAS}

\subsection{Proof of Lemma \ref{lem:sig1}}
In order to find the lower bound on $|\Upsilon_1|$, we shall consider $\Upsilon_1'= \{ \sigma\in \Upsilon_1:   \emptyset \in \sigma\}$; i.e., the set of  full chains in $Q_m$  containing
only $\emptyset$ and no other sets from $\calg$.

\noindent
Recall that $\calg =  \{\emptyset\} \cup \calt \cup \calu $, where $\calt$ are the minimal elements of $\calg - \{\emptyset\}$, and $\calt = \calt_{1} \cup \calt_{2}\cup \cdots$, where $\calt_{i}$ is the family of sets from $\calt$ of size $i$.  Without loss of generality, the   one-element members  of $\calt$ are  $\{\eta+1\}, \{\eta+2\}, \ldots, \{m\} $,
which correspond to $1$-element subsets of  $[m]-[\eta]$. A graph $G$ is defined on vertex set $[\eta]$ with edges corresponding to sets from $\calt_2$,  $e=|E(G)|, \overline{e} = |E(\overline{G})|$,
$d(v), \overline{d}(v)$ is the degree of $v$ in $G$ and $\overline{G}$, respectively.

\noindent
Let $x\in [\eta], t \in [m]-[\eta]$.

\noindent
If $\{x,t\} \in \calg$ then denote $\calc_1(x, t)$  to be the set  of full   chains of the form $\emptyset, \{x\}, \{x,y\}, \{x,y,t\},\ldots$, where $y \in [\eta]$,  $\{x,y\}\not\in \calg$.
We have that $\calc_1(x, t) \subseteq \Upsilon_1'$ and $|\calc_1(x, t) | \geq (m-3)!\overline{d} (x)$.

\noindent
If $\{x, t\} \not\in \calg$ then denote $\calc_2(x,t) $ to be the set of full   chains of the form $\emptyset, \{x\}, \{x,t\},\ldots$  unless such a chain passes through $A\cup\{t\}$ where $A\cup\{t\}\in \calg$.
We have that $\calc_2(x, t) \subseteq \Upsilon_1'$ and $|\calc_2(x, t) | \geq (m-2)! -   (m-3)!\overline{d} (x)$.

\noindent
Observe also for any $t, t' \in [m]-[\eta]$,  $\calc_1(x,t) \cap \calc_1(x,t') = \emptyset$ and  $\calc_2(x, t)\cap \calc_2(x, t') = \emptyset$.
Thus for each $x\in [\eta]$, the number of chains in $\Upsilon_1'$ passing through $x$ is at least
$$\sum_{\scriptsize \begin{array}{l} t\in[m]-[\eta] \\ \{x,t\}\in\calg\end{array}}|\calc_1(x,t)|+\sum_{\scriptsize \begin{array}{l} t\in[m]-[\eta] \\ \{x,t\}\not\in\calg\end{array}}|\calc_2(x,t)| \geq  \sum_{t\in [m]-[\eta] }\min \{ (m-3)! \overline{d}(x),   (m-2)! -   (m-3)!\overline{d} (x)\}.$$
  Thus,
\begin{eqnarray}
    |\Upsilon_1| \geq |\Upsilon_1'|
    & \geq &   \sum_{x\in [\eta]} \sum_{t\in [m]-[\eta] } \min \{ (m-3)! \overline{d} (x),   (m-2)! -   (m-3)!\overline{d} (x)\}  \nonumber\\
      & = & (m-\eta)(m-3)!\left(2\ove -\sum_{x\in [\eta]}\max\left\{0,2\overline{d} (x)-m+2\right\}\right).\label{star}
         \end{eqnarray}
\noindent
Consider the  set $D$ of all sequences of $\eta$ nonnegative real numbers which are at most $\eta-1$,  and which add up to $2\ove$.
Note that the degree sequence of $\overline{G}$ is in $D$. Thus,
\begin{eqnarray}
\sum_{x \in [\eta]}\max \{0,2\overline{d} (x) - m+2\}
& \leq &  \max_{(d_1, \ldots, d_n)\in D}  \sum _{i=1}^{\eta}  \max \{0, 2d_i - m + 2\}\nonumber \\
& \leq & \frac{2\ove}{\eta-1} (2\eta -2 -m +2).\nonumber
\end{eqnarray}
\noindent
Returning to (\ref{star}), and recalling that $a= \eta/m$, and
\begin{eqnarray*}
 \mu=\begin{cases}
 1,& a < 1/2, \\
 \frac{1-a}{a},&   a \geq 1/2,
 \end{cases}
\end{eqnarray*}
 we have
\begin{eqnarray*}
(m-\eta)(m-3)!\left(2\ove -\sum_{x\in [\eta]}\max\left\{0,2\overline{d} (x)-m+2\right\}\right)  &\geq &
\begin{cases}
 (m-\eta)(m-3)!2\ove,  &  2\eta \leq m-2; \nonumber\\
  (m-\eta)(m-3)!   2\ove \left[  \frac{m-\eta-1}{\eta -1}  \right], & 2\eta > m-2\nonumber
  \end{cases}\\
  & \geq &
  (m-\eta)(m-3)!2\ove(\mu-O(m^{-1})).
     \end{eqnarray*}

 \noindent
 Therefore,  since $b= e/ \binom{\eta}{2}$,
 \begin{eqnarray*}
|\Upsilon_1| & \geq &  (m-\eta)(m-3)! 2\ove(  \mu- O(m^{-1}))\\
& = & m! \frac{m-\eta}{m (m-1)(m-2)} \eta^2 (1-b) (\mu -O(m^{-1})) \\
& = & m!\left[   b(a^3-a^2)\mu  + (a^2-a^3)\mu  - O(m^{-1})\right].
\end{eqnarray*}\qed


\subsection{Proof of Lemma \ref{sigma3}}
 For each $T\in\calt$, let $\calu_T=\{U\in\calu : U\supset T\}$ and let
 $$\calu_T'=\{V\supset T : ~ |V|=|T|+1,  ~\not\exists T_0\in\calt-\{T\}, ~T_0\subset V\}.$$
 We say that $\calg'=\emptyset\cup\calt\cup\bigcup_{T\in\calt}\calu_T'$ is a \textbf{compressed family}.

\noindent
We have that $|\Upsilon_m^3(\emptyset, \calg) | \leq |\Upsilon_m^3(\emptyset, \calg')|$.
Indeed, if a chain contains both $T\in\calt$ and $U\in\calu$, then there is some $U'\in\calu_T'$ that this chain contains also.

\noindent
Let  $\Upsilon_3' = \Upsilon_m^3(\emptyset, \calg')$.
Recall that $\calt = \calt_1\cup \calt_2 \cup \cdots$, where $\calt_i$ is the family of sets from $\calt$ of size $i$.
To bound $| \Upsilon_3'|$, we count first the number of full chains from $\Upsilon_3'$  containing sets from $\calt_1$,  then those containing sets from $\calt_2$,
and finally those containing sets from $\calt_i$, $i\geq 3$.

\noindent
Recall that the graph $G$ is defined on vertex set $[\eta]$ with edges corresponding to sets from $\calt_2$.
Let $\alphaone$ be the number of triples from $[\eta]$ which  induce exactly one edge in $G$.
 Let $B_0$ be the set of $4$-element sets from $[\eta]$ which do not induce an edge in $G$, and let $\beta_0= |B_0|$.

\noindent
There are at most $(m - \eta)\eta(m-2)!$ chains from $\Upsilon_3'$ containing sets from $\calt_1$.
There are $2\alphaone (m-3)!$ such chains containing sets from $\calt_2$.
We need to do some more work to bound the number of chains from $\Upsilon_3'$ containing sets from $\calt_i$, $i\geq 3$.
Call the set of such chains $Y$.

\noindent
 Recall that if $T, U \in \calg'$,  $T\subseteq U$, $T\neq \emptyset$, then
the number of full chains through $T$ and $U$ is $|T|!(|U-|T|)!(m-|U|)!$.
Since $\calg'$ is a compressed family, we have that $|T|=|U|-1$.
Moreover, if $T, U$ belong to a chain in $Y$, we have that $|U| \geq 4$.  Let
$$\calu^* = \{ U\in \calg'\setminus\calt : U \in C\in Y\}.$$
\noindent
Since for each $U \in \calg'$ there is at most one $T \in \calg'$, $T\neq \emptyset$ such that $T\subseteq U$, we have
$$|Y| = \sum_{U\in \calu^*}  \sum_{T\subseteq U,  T\in \calt}  |T|!(m-|U|)!\leq \sum_{U\in \calu^*} (|U|-1)!(m-|U|)!  =  \sum_{U\in \calu^*} \frac{1}{|U|} |U|!(m-|U|)! \leq  \frac{1}{4} \sum_{U\in \calu^*}  |U|!(m-|U|)!.$$
 The last  summation counts the number of full chains   containing a set from $\calu^*$.
Since for each $U \in \calu^*$, there is $B \in B_0$, $B\subseteq U$, we have
 that the number of  full chains passing through sets in  $\calu^*$ is at most the
number of full chains passing through  $B_0$ sets. Thus
     $$|Y|  \leq \frac{1}{4} \sum_{B\in B_0}  |B|!(m-|B|)!\leq  \frac{1}{4} \sum_{B\in B_0}  |4|!(m-|4|)! = \frac{1}{4} 4!(m-4)! |B_0| \leq 3!(m-4)! \beta_0. $$

\noindent
So, we have that 
\begin{equation} \label{UB-3}
|\Upsilon_3| \leq |\Upsilon_3'| \leq  (m-\eta)\eta (m-2)! + 2 (m-3)! \alphaone+ 3!(m-4)! \beta_0.
\end{equation}

\noindent
To bound the last two terms, we use the following lemma.

\begin{lem}\label{lem:graph}  With $\alpha_1$ and $\beta_0$ defined as above for $G$,  an  $\eta$-vertex graph, and $a=\eta/m$,
$$ \alpha_1+\frac{3}{m-3}\beta_0 \leq\frac{\eta^3}{8} +\frac{e^2}{\eta}(a-2) +\frac{1}{4}e\eta(4-3a) -\frac{1}{8}(1-a)\eta^3 +O(m^2).$$
\end{lem}


\begin{proof}
Let $\betai=\betai(G)$ be the number of $4$-element subsets of the vertex set of $G$ spanning exactly $i$ edges.  Moreover, let
$$\betatwo(G)=\betapar(G)\cup\betawedge(G), \quad   \betathree(G)=\betatri(G)\cup\betaw(G)\cup\betapath(G),$$    where
$\betapar$, and $\betawedge$ count  such subsets inducing  two disjoint edges, and  two adjacent edges, respectively;
$\betatri$,  $\betaw$, and $\betapath$  count   the number of such subsets inducing triangle,  a star of three edges, and a path with three edges,  respectively.
We also denote ${\Obeta_i}(G)= \beta_i(\overline{G})$,   ${ \Obetapar}(G) = \betapar(\overline{G})$,  ${\Obetawedge}(G) = \betawedge(\overline{G})$,
${\Obetatri}(G)= \betatri(\overline{G})$,   ${  \Obetaw}(G) = \betaw(\overline{G})$,   ${ \Obetapath}(G) = \betapath(\overline{G})$.
Note that ${\Obetatri}(G)= \betaw (G)$,   ${\Obetapath}(G) = \betapath(G)$.
Let $d(v)$ and $\overline{d}(v)$ be the degree of $v$ in $G$ and in $\overline{G}$, respectively. Then,
 \begin{eqnarray*}
\sum_{v\in V} d(v) \binom{\overline{d}(v)}{2} & = &
2\betaone+4\betapar+2\betawedge+2\betapath+3\betaw+\Obetawedge, \label{eq:dee} \nonumber\\
\sum_{v\in V} \binom{d(v)}{3} & = &
\betaw+\Obetawedge+2\Obetaone+4\Obetazero, \label{eq:ddd} \nonumber\\
\sum_{v\in V} \overline{d}(v) \binom{d(v)}{2} & = &
2\Obetaone+4\Obetapar+2\Obetawedge+2\Obetapath+3\Obetaw+\betawedge,   \label{eq:edd} \nonumber\\
\sum_{v\in V} \binom{\overline{d}(v)}{3} & = &
 \Obetaw+\betawedge+2\betaone+4\betazero. \label{eq:eee}\nonumber
\end{eqnarray*}

\noindent
Observe that the last two equations are complementary of the first two.  In order  to bound $\alphaone+\frac{3}{m-3}\betazero$,
we shall express everything  in terms of $\beta$s,  and then in terms of $e$.

\noindent
Since
$$ \alpha_1(\eta-3)=2\betaone+4\betapar+2\betawedge+3\betatri+2\betapath+\Obetawedge, $$
and
$$ \binom{\eta}{4} = \beta_0+\beta_1+\beta_2+\beta_3+\beta_4+\beta_5+\beta_6
=\beta_0+\beta_1 +(\betapar+\betawedge) +(\betatri+\betapath+\betaw) +(\Obetapar+\Obetawedge)+\Obetaone+\Obetazero,$$
we have, by recalling that $a={\eta}/{m}>(\eta-3)/(m-3)$,

\begin{eqnarray*}
Q & := & (\eta-3)\left[\alphaone+\frac{3}{m-3}\betazero\right] \\
& = &  3\frac{\eta-3}{(m-3)}\betazero+2\betaone+4\betapar+2\betawedge+3\betatri+2\betapath+\Obetawedge \\
& < & 3\binom{\eta}{4} +(3a-3)\betazero +(-1)\betaone +1\betapar+(-1)\betawedge +0\betatri+(-1)\betapath+(-3)\betaw \\
& & +(-2)\Obetawedge+(-3)\Obetapar +(-3)\Obetaone +(-3)\Obetazero .
\end{eqnarray*}

\noindent
At the price of slightly increasing the right hand side, we collect the terms in order to utilize the various formulas that sum the degrees:
\begin{eqnarray*}
%
%
%
Q & \leq & 3\binom{\eta}{4} +\dfrac{1}{4}\left[2\betaone +4\betapar +2\betawedge +2\betapath +(1-3) \Obetawedge -6\Obetaone -12\Obetazero\right] \\
& & +\dfrac{(1-a)}{4}\left[2\Obetaone +4\Obetapar+2\Obetawedge +2\betapath+(1-3)\betawedge -6\betaone -12\betazero\right] \\
& = & 3\binom{\eta}{4} +\frac{1}{4}\left[\sum_v d(v)\binom{\overline{d}(v)}{2}-3\sum_v \binom{d(v)}{3}\right] +\frac{1-a}{4}\left[\sum_v \overline{d}(v)\binom{d(v)}{2}-3\sum_v \binom{\overline{d}(v)}{3}\right] \\
& = & 3\binom{\eta}{4} +\frac{\eta-3}{8}\left[\sum_v d(v)(\eta-2d(v))\right] +(1-a)\frac{(\eta-3)}{8}\left[\sum_v \overline{d}(v)(\eta-2\overline{d}(v))\right] .
%
%
\end{eqnarray*}

\noindent
We can use the fact that $\overline{d}(v)=\eta-d(v)-1$ and collect terms
$$ Q \leq
3\binom{\eta}{4} +\frac{\eta-3}{8}(-4+2a)\sum_v d^2(v) +\frac{\eta-3}{8}\left(4\eta-4-3a\eta+4a\right) \sum_v d(v) +\frac{\eta-3}{8}(a-1)\eta(\eta-1)(\eta-2) . $$

\noindent
Using the fact that $\sum_v d(v)=2e$ and $\sum_v d^2(v) \geq 4e^2/\eta$,
$$ Q\leq\frac{\eta^4}{8} +e^2(a-2) +\frac{1}{4}e\eta^2(4-3a) -\frac{1}{8}(1-a)\eta^4 + O(\eta^3). $$

\noindent
Dividing $Q$ by $\eta-3$ and observing that $\eta\leq m$, this concludes the proof of Lemma \ref{lem:graph}.
\end{proof}

\noindent
Now, we return to the upper bound  $(\ref{UB-3})$  on $|\Upsilon_3|$, recalling that $a= \eta/m $ and $ e = b \binom{\eta}{2}$,
$$ |\Upsilon_3|\leq (m-\eta)\eta(m-2)! +2(m-3)!\alphaone +3!(m-4)!\beta_0 . $$

\noindent
Because of Lemma~\ref{lem:graph},
\begin{eqnarray*}
|\Upsilon_3| & \leq & m!\left[a(1-a) +\frac{2}{m^3}\left(\frac{\eta^3}{8} +\frac{e^2}{\eta}(a-2) +\frac{1}{4}e\eta(4-3a) -\frac{1}{8}(1-a)\eta^3 +O(m^{-1})\right)\right]\\
& = & m!\left[b^2\left(a^4/2-a^3\right) +b\left(a^3-3a^4/4\right) +\left(a^4/4-a^2+a\right) +O(m^{-1})\right].
\end{eqnarray*}
This concludes the proof of Lemma \ref{sigma3}.
\qed

\subsection{Proof of Lemma~\ref{lem:calc}}
The estimations here can be checked by a symbolic manipulation program.

\noindent
Set
$$ Q':=b^2\left(a^4/2-a^3\right)+b\left(a^3-3a^4/4-a^3\mu+a^2\mu\right) +\left(a^4/4-a^2+a-a^2\mu+a^3\mu\right) . $$

\noindent
If $ 0\leq a<1/2$ we have that $\mu=1$ and
$$ Q'=b^2\left(a^4/2-a^3\right) +b\left(-3a^4/4+a^2\right) +\left(a^4/4+a-2a^2+a^3\right)\leq 0.25 , $$
which is achieved when $b=1$ and $a=1/2$.

\noindent
If $1/2\leq a\leq 1$ we have that $\mu=(1-a)/a$ and
$$ Q'=b^2\left(a^4/2-a^3\right) +b\left(2a^3-3a^4/4-2a^2+a\right) +\left(a^4/4+a^2-a^3)\right)<0.283261 . $$
The maximum is achieved when $a\approx 0.935$ and $b\approx 0.285$.\qed

\section{Conclusions}
The method we use is local, it allows us  to count the number of full chains with three or one element in $\calf$.
Using this method, one could not get a bound better than $2.25N$ for $ex(n, Q_2)$.
To see this, consider a set system with elements from $[m]$, where $m$ is even and
$[m] = M_1\cup M_2$, $ M_1 = [m/2], M_2 =  \{m/2+1, m/2+2, \ldots, m\}$.
$\calg =  \{\emptyset\}\cup \calt \cup \calu$, where
$$\calt = \binom{M_1}{2}\cup  \binom{M_2}{2}, $$
$$\calu =  \{ \{a, b, c\} :  a, b \in M_2, c\in M_1 \} \cup \{ \{a,b,c\}: a, b \in M_1, c\in M_2\}.$$
We have that the number of full chains in $Q_m$  containing three elements of $\calg$  is
at  $4\binom{m/2}{2}m/2 (m-3)!$. On the other hand, each full chain contains at least one nonempty set from $\calf$.
Thus $|\Upsilon_3|\geq m!/4$.

\section{Acknowledgments}
The authors are grateful to Jerry Griggs and Roger Maddux for comments on the proof.  We also thank anonymous referees for their hard work and careful reading.

\bibliographystyle{plain}
\bibliography{myrefs}

\end{document}